\documentclass{article}
\usepackage{amsmath,amsthm,amssymb,mathrsfs,amstext, titlesec,enumitem,comment}
\usepackage{hyperref}
\makeatletter

\titleformat{\section}
{\normalfont\Large\bfseries}{\thesection}{1em}{}
\titleformat*{\subsection}{\Large\bfseries}
\titleformat*{\subsubsection}{\large\bfseries}
\titleformat*{\paragraph}{\large\bfseries}
\titleformat*{\subparagraph}{\large\bfseries}

\renewcommand{\@seccntformat}[1]{\csname the#1\endcsname. }

\renewenvironment{abstract}{%
    \if@twocolumn
      \section*{\abstractname}%
    \else %% <- here I've removed \small
      \begin{center}%
        {\bfseries \Large\abstractname\vspace{\z@}}%  %% <- here I've added \Large
      \end{center}%
      \quotation
    \fi}
    {\if@twocolumn\else\endquotation\fi}
    
\makeatother
\theoremstyle{plain}
\newtheorem{thm}{Theorem}[section]
\newtheorem{prop}{Proposition}

\theoremstyle{definition}
\newtheorem{defn}[thm]{Definition}

\newtheorem{cor}[thm]{Corollary}
\newtheorem{qn}[thm]{Question}

\theoremstyle{numberfirst}

\newcommand{\N}{\mathbb{N}}

\newcommand{\F}{\wp_{fin}\left(\mathbb{N}\right)\setminus\{\emptyset\}}

\date{\vspace{-5ex}}

\begin{document}

\title{Polynomial extension of the Stronger Central Sets Theorem}

\author{ 
Sayan Goswami\\  \textit{sayangoswami@imsc.res.in}\footnote{The Institute of Mathematical Sciences, A CI of Homi Bhabha National Institute, CIT Campus, Taramani, Chennai 600113, India.}\\
\and
Lorenzo Luperi Baglini \\ 
\textit{lorenzo.luperi@unimi.it}
\footnote{Dipartimento di Matematica, Università di Milano, Via Saldini 50, 20133 Milano, Italy.}		
\and	
	Sourav Kanti Patra\\ \textit{spatraskpatra.math@dde.buruniv.ac.in}\footnote{ Department of Mathematics, Centre for Distance and Online Education, University of Burdwan, Burdwan-713104,West Bengal.}
	}

\maketitle

\begin{abstract}

In \cite{key-11} Furstenberg introduced the notion of central set and established his famous Central Sets Theorem. Since then, several improved versions of Furstenberg's result have been found. The strongest generalization has been published by De, Hindman and Strauss in \cite{key-5}, whilst a polynomial extension by Bergelson, Johnson and Moreira appeared in \cite{key-1}.

In this article, we will establish a polynomial extension of the stronger version of the central sets theorem, and we will discuss properties of the families of sets that this result leads to consider. 
\end{abstract}

\section{Introduction}

A core problem in Ramsey Theory over the naturals is the characterization of which families $\mathcal{F}$ of subsets of $\mathbb{N}$ are partition regular, \textit{i.e.} which families have the property that whenever $\mathbb{N}=\bigcup_{i=1}^r A_i$ is a finite partition of $\mathbb{N}$, at least one of the $A_i$'s belongs to $\mathcal{F}$. When the family $\mathcal{F}$ has the property that whenever any $A\in\mathcal{F}$ is finitely partitioned one of the pieces in the partition belongs to $\mathcal{F}$, the family is said to be strongly partition regular. 

Two fundamental and classical results in Ramsey theory state, respectively, that the family of sets that contain arbitrarily long arithmetic progressions (called AP-rich sets) and the family of sets that contain an infinite subset $X$ and all the finite sums of distinct elements of $X$, called IP-sets, are partition regular\footnote{It has been proven that they are indeed strongly partition regular.}. The first result is called van der Waerden's theorem \cite{key-8}, the latter is called Hindman's theorem \cite{key-9}.

In his seminal work \cite{key-11}, Furstenberg used methods and notions from topological dynamics to define the notion of the central set and provided a joint extension of both van der Waerden's and Hindman's theorems, known as the Central Sets Theorem.

\begin{thm} 
\label{cst} \textup{[Central Sets Theorem]}
Let $l\in\mathbb{N}$, and $A\subseteq\mathbb{N}$ be a central set. For
each $i\in\{1,2,\ldots,l\}$ let $\langle x_{i,m}\rangle_{m=1}^{\infty}$ be a sequence in $\mathbb{\mathbb{N}}$. Then there exists a sequence
$\langle b_{m}\rangle_{m=1}^{\infty}$ in $\mathbb{N}$ and $\langle K_{m}\rangle_{m=1}^{\infty}$
in $\wp_{fin}\left(\mathbb{N}\right)\setminus\{\emptyset\}$ such that
\begin{enumerate}
\item For each $m$, $\max K_{m}<\min K_{m+1}$ and
\item For each $i\in\{1,2,\ldots,l\}$ and $H\in\wp_{fin}\left(\mathbb{N}\right)\setminus\{\emptyset\}$,
$\sum_{m\in H}(b_{m}+\sum_{t\in K_{m}}x_{i,t})\in A$.
\end{enumerate}
\end{thm}

The original definition of the notion of central set was dynamical; however, the following equivalent simpler ultrafilters\footnote{In this paper we assume the reader to know the basics of the algebra of $\beta\mathbb{N}$.} characterization was found in \cite{key-16}.

\begin{defn} $A\subseteq \mathbb{N}$ is central if $A$ belongs to a minimal idempotent $\mathcal{U}\in\beta\mathbb{N}$. \end{defn}

Several generalizations of Theorem \ref{cst} to semigroups have been found in the literature; for details, we refer to \cite{key-12}. As we are interested to provide a new general version of Theorem \ref{cst} for\footnote{Most of our proofs could easily be generalized to the case of countable abelian semigroups. } $\mathbb{N}$, we will recall the specification to $\mathbb{N}$ of some of these generalizations. 
At the best of our knowledge, the following result of De, Hindman and Strauss, published in \cite{key-5}, is the most general commutative version of the Central Sets theorem till date.

\begin{thm}
 \label{gcst} \textup{[Stronger Central Sets Theorem]} Let $\tau=\mathbb{N}^{\mathbb{N}}$ and let $C \subseteq S$ be central. There exists functions $\alpha : \wp_{fin}(\tau)\to \mathbb{N}$ and $H: \wp_{fin}(\tau) \to \wp_{fin}\left(\mathbb{N}\right)\setminus\{\emptyset\}$ such that 
\begin{enumerate}
\item \label{1.41}if $F,G \in \wp_{fin}(\tau)$ and $F \subsetneqq G$ then $\max H(F) < \min H(G)$, and 
\item \label{1.42}whenever $m \in \mathbb{N}$, $G_1, G_2, \ldots , G_m \in \wp_{fin}(\tau)$, $G_1 \subsetneq G_2 \subsetneq \cdots \subsetneq G_m$ and for each $i \in \{1,2, \ldots , m\}$, $f_i \in G_i$, one has $$\sum_{i=1}^{m}\big(\alpha(G_i)+\sum_{t\in H(G_i)}f_i(t)\big)\in C.$$
\end{enumerate}
\end{thm}

We are interested to arrive to a nonlinear version of Theorem \ref{gcst}. Nonlinear versions of linear statements in combinatorics are usually much harder to obtain. They usually involve the set of polynomials with integer coefficients that vanish at 0 and send $\mathbb{N}$ into $\mathbb{N}$; we will denote this set by $\mathbb{P}$ from now onwards. 

The nonlinear version of van der Waerden's theorem, known as the Polynomial van der Waerden's Theorem\footnote{Actually, the authors proved a generalized version of the much stronger Sz\'emeredi's Theorem, but we are not going to discuss it in this paper.} was established by Bergelson and Liebman in \cite{key-9.1}, using the methods of topological dynamics and PET induction.

\begin{thm} \label{pvdw} 
\textup{[Polynomial van der Waerden theorem] \cite{key-9.1}}
Let $r\in \mathbb{N},$ and $\mathbb{N}=\bigcup_{i=1}^r C_i$ be a $r$-coloring of $\mathbb{N}$. Then for any $F\in \wp_{fin}(\mathbb{P})$, there exist $a,d\in \mathbb{N}$ such that for all $p\in F$, $a+p \left(d\right) \in C_j$,  for some $1\leq j\leq r$.
\end{thm}

A natural question arises whether one can provide a joint extension of Polynomial van der Waerden's and Hindman's theorems. This was established by Bergelson, Johnson and Moreira in \cite{key-1}, and it is now known as the Polynomial Central Sets Theorem. Whilst their version of the Theorem involves arbitrary countable groups, we will specialize it here to $\mathbb{N}$; before stating it, we need to recall some definitions from \cite{key-1}.

%FORSAYAN: Here you wrote the definitions for groups. Check if my translation to $\mathbb{N}$ works and is ok with respect to the original source.

\begin{defn} \label{reqdefn} 

\begin{enumerate} Let $p\in\beta\mathbb{N}$ be an ultrafilter and let $\Gamma\subseteq\mathbb{N}^{\mathbb{N}}$. 

\item \label{R-f}($R$-family)  We say that $\Gamma$ is an $R$-family with respect to $p$ if for every finite set $F\subseteq\Gamma$, every $A\in p$ and every IP-set $\left\langle y_{\alpha}\right\rangle _{\alpha\in \wp_{fin}\left(\mathbb{N}\right)}$, there exist $x\in \mathbb{N}$ and $\alpha\in F$ such that $x+f\left(y_{\alpha}\right)\in A,\,\forall f\in F$.

\item \label{licit} (Licit) We say that $\Gamma$ is licit if for any $f\in\Gamma$ and any $z\in \mathbb{N}$, there exists a function $\phi_{z}\in\Gamma$ such that $f\left(y+z\right)=\phi_{z}\left(y\right)+f\left(z\right)$.

\item \label{IP-regular} An endomorphism $c\in \mathbb{N}\rightarrow \mathbb{N}$ is called $IP$-regular
if for every $IP$-set $\left\langle x_{\alpha}\right\rangle _{\alpha\in\wp_{fin}\left(\mathbb{N}\right)}$ there exists an IP-set $\left\langle y_{\alpha}\right\rangle _{\alpha\in\wp_{fin}\left(\mathbb{N}\right)}$ such that $\left\langle c\left(y_{\alpha}\right)\right\rangle _{\alpha\in\wp_{fin}\left(\mathbb{N}\right)}$
is a sub-IP-set of $\left\langle x_{\alpha}\right\rangle _{\alpha\in\wp_{fin}\left(\mathbb{N}\right)}$.
\end{enumerate}
\end{defn}

Specialized to $\mathbb{N}$, the Polynomial Central Sets Theorem reads as follows.

\begin{thm} 
 \label{PCST} \textup{[Polynomial Central Sets Theorem]} Let $p\in\beta \mathbb{N}$ be an idempotent
ultrafilter, let $\Gamma\subseteq\mathbb{N}^{\mathbb{N}}$ be an $R$-family with respect to $p$
which is licit. Then for any finite set $F\subseteq\Gamma$, any $A\in p$
and any $IP$-set $\langle y_{\alpha}\rangle_{\alpha\in \wp_{fin}\left(\mathbb{N}\right)}$, there exist a sub-$IP$-set $\langle z_{\beta}\rangle _{\beta\in \wp_{fin}\left(\mathbb{N}\right)}$ of $\langle y_{\alpha}\rangle_{\alpha\in \wp_{fin}\left(\mathbb{N}\right)}$
and an $IP$-set $\langle x_{\beta}\rangle_{\beta\in \wp_{fin}\left(\mathbb{N}\right)}$ such that for all $f\in F$ and for all $\beta\in\wp_{fin}\left(\mathbb{N}\right)$,
\[
x_{\beta}+f\left(y_{\beta}\right)\in A.
\]
\end{thm}

Inspired by the above result, our goal in this paper is to provide a polynomial extension, in the flavour of Theorem \ref{PCST}, of the Stronger Central Sets Theorem \ref{gcst}. This will be done in Section 2, where new special classes of sets related to our result, called $J_{p}-$ and $C_{p}-$sets, will be introduced. In Section 3, we will provide equivalent characterizations of the notions of $J_{p}-,C_{p}-$sets in terms of nonstandard analysis. Finally, in Section 4 we will discuss some open problems that arise as consequences of our main result.

\section{Polynomial extension of the stronger Central Sets Theorem}

Since for any $j\in \mathbb{N},$ $\mathbb{N}^j$ is piecewise syndetic in $\mathbb{Z}^j$,  it follows from \cite[Corollary 2.3]{key-13} that any set $A\subseteq \mathbb{N}$ which is piecewise syndetic in $\mathbb{N}$, is also piecewise syndetic in $\mathbb{Z}$.
Hence the following version of the IP-Polynomial Van der Waerden's theorem (that we specialize here to $\mathbb{N}$) is a special case of \cite[Corollary 2.12]{key-1}. 

\begin{thm}\label{BJM}
\textup{(Abstract IP-Polynomial van der Waerden theorem)}
Let $j\in \mathbb{N} $ and  $A\subseteq \mathbb{N}$ be a piecewise syndetic set. Then for any finite set of polynomials $F$ from $\mathbb{N}^j$ to $\mathbb{N}$ \footnote{Polynomials from  $\mathbb{N}^j$ to $\mathbb{N}$ are multidimensional polynomials.} and any IP-set $\left(x_{\alpha}\right)_{\alpha\in\wp_{fin}(\mathbb{N})}$,  there exists  $a\in \mathbb{N}$ and $\beta \in \wp_{fin}\left(\mathbb{N}\right)\setminus\{\emptyset\}$ such that 

\[
a+f(x_{\beta})\in A
\]
for all $\beta\in\wp_{fin}\left(\mathbb{N}\right)\setminus\{\emptyset\},f\in F$.
\end{thm}

The following simple consequence of the above theorem will motivate us to introduce the notion of $J_p$ set.

%FORSAYAN: the following result can be done for $\mathbb{N}$ instead of $\mathbb{Z}$, right? What is the reason to do it for $\mathbb{Z}$?

\begin{thm}
\label{recurrence} Let $l\in \mathbb{N}$ and $A\subset \mathbb{N}$ be a piecewise syndetic set. For $i=1,2,\ldots,l$, let $\left( x_{\alpha}^i \right)_{\alpha\in\F}$ be an IP-set. Then for all finite $F\in \wp_{fin}(\mathbb{P})$, there exist $a\in \mathbb{N}$, $\beta \in \wp_{fin}\left(\mathbb{N}\right)\setminus\{\emptyset\}$ such that 
\[
a+P(x^i_\beta)\in A
\]
for all $i\in \{1,2,\ldots,l\}$ and $P\in F$.
\end{thm}

\begin{proof}
Consider the IP set $\left( x_\alpha^1, x_\alpha^2,\ldots,x_\alpha^l \right)_{\alpha}$ in $\mathbb{N}^l$. For $i\in \{1,2,\ldots,l\}, P\in F$ define $f_P^i(x_1,\ldots,x_l):=P(x_i),$
and let $F_{i}=\{f_{P}^{i}\left(x_{1},\ldots,x_{l}\right)\mid P\in F\}$. Finally, let $G=\bigcup_i F_i$. Then $G$ is a finite set of polynomials from $\mathbb{N}^l$ to $\mathbb{N}$ that vanish at $0$. Applying Theorem \ref{BJM} we find $a\in \mathbb{N}$ and $\beta \in \wp_{fin}\left(\mathbb{N}\right)\setminus\{\emptyset\}$ such that
\[
a+f_{P}^i\left(x_\beta^1, x_\beta^2,\ldots,x_\beta^l\right)=a+P\left(x_\beta^i\right)\in A
\]
for all $i\in \{1,2,\ldots,l\}$ and $P\in F$, as desired.
\end{proof}

%FORSAYAN: if we do Thm \ref{recurrence} for $\mathbb{N}$, there is no need for this corollary.

%\begin{rem}
%As the set $\mathbb{N}$ is a thick set in $\mathbb{Z}$, any piecewise syndetic set in $\mathbb{N}$ is also piecewise syndetic in $\mathbb{Z}$, follows from \cite[Corollary 2.3]{key-13}.
%\end{rem}

Theorem \ref{recurrence} leads to strengthen polynomially the notion of $J$-set (that we recall) and to introduce that of $J_{p}$-set.

\begin{defn}
A set $A\subseteq\mathbb{N}$ is called a
$J$-set if for all $l\in\mathbb{N}$ and
for all IP-sets $(x_{\alpha}^i)_{\alpha\in\wp_{fin}\left(\mathbb{N}\right)\setminus\{\emptyset\}}, i=1,\ldots,l,$ there exist $a\in\mathbb{N}\cup\{0\}$ and $\beta\in\wp_{fin}\left(\mathbb{N}\right)\setminus\{\emptyset\}$
such that 
\[
a+\left(x_{\beta}^i\right)\in A
\]
for all $i\in \{1,2,\ldots,l\}$.

A set $A\subseteq\mathbb{N}$ is called a
$J_{p}$-set if for all finite $F\subset\mathbb{P}$, for all $l\in\mathbb{N}$ and
for all IP-sets $(x_{\alpha}^i)_{\alpha\in\wp_{fin}\left(\mathbb{N}\right)\setminus\{\emptyset\}}, i=1,\ldots,l,$ there exist $a\in\mathbb{N}\cup\{0\}$ and $\beta\in\wp_{fin}\left(\mathbb{N}\right)\setminus\{\emptyset\}$
such that 
\[
a+P\left(x_{\beta}^i\right)\in A
\]
for all $P\in F$ and $i\in \{1,2,\ldots,l\}$.

\end{defn}

By the definition, it trivially holds that a $J_{p}$ set is a $J$ set. We discuss the converse in Section 4.

The family of $J_{p}$ set is actually quite rich. To prove this, let us recall the notion of (upper) Banach density:

\begin{defn} A set $A\subseteq \mathbb{N}$ has positive (upper) Banach density if \[\limsup_{n\rightarrow +\infty} \max_{m\in\mathbb{N}}  \dfrac{\vert A\cap [m+1,\dots,m+n] \vert}{\vert  n \vert}>0.\]\end{defn}

The following result, known as the Multidimensional IP polynomial Szemer\'{e}di theorem (which, once again, we state only for $\mathbb{N}$), will entail that sets with positive Banach density are $J_{p}$-sets.
%FORSAYAN: check that it is ok, can be seen as a density refinement of Theorem \ref{recurrence}.

\begin{thm}\label{n}
Let $B\subseteq \mathbb{N}$ have positive upper Banach density. Let $\langle y_{\alpha} \rangle_{\alpha \in \wp_{fin}\left(\mathbb{N}\right)\setminus\{\emptyset\}}$ be an IP-set. For any finite family $F\subset \mathbb{P}$ there exists $x\in \mathbb{N}^n$ and $\alpha \in \wp_{fin}\left(\mathbb{N}\right)\setminus\{\emptyset\}$ such that $x+f(y_\alpha)\in B$ for all $f\in F.$
\end{thm}

The proof of the following corollary is verbatim the proof of theorem \ref{recurrence}, where one has to apply theorem \ref{n} instead of theorem \ref{BJM} with $j=l$. Hence we omit the proof.

\begin{cor}
\label{recurrence1} Let $l\in \mathbb N$ and $A\subset \mathbb{Z}$ have positive upper banach density. For each $i=1,2,\ldots,l$, let $\left( x_{\alpha}^i \right)_{\alpha\in\F}$ be an IP-set. Then for all finite $F\in \wp_{fin}(\mathbb{P})$, there exist $a\in \mathbb{N}$, $\beta \in \wp_{fin}\left(\mathbb{N}\right)\setminus\{\emptyset\}$ such that 
\[
a+P(x^i_\beta)\in A
\]
for all $i\in \{1,2,\ldots,l\}$ and $P\in F$.
\end{cor}

In particular, Corollary \ref{recurrence} proves that any set with a positive Banach density is a $J_{p}$-set. However, the converse does not hold. This can be proven using the same example that Hindman produced in \cite[Theorem 2.1]{key-14} to show the existence of $J$-sets of density $0$. Let us recall the construction of his example (we refer to \cite[Theorem 2.1]{key-14} for details). For $n\in \mathbb{N}$, let 

\[a_n=\min\left\{t \in \mathbb{N}\mid\left(\frac {2^n-1}{2^n}\right)^t\leq \frac{1}{2}\right\},\]
and let $S_n=\sum_{i=1}^{n} a_{i}$. Let $b_0=0$, let $b_1=1$, and for $n\in \mathbb{N}$ and
$t\in \left\{S_{n}, S_{n+1},\dots,S_{n+1}-1\right\}$, let $b_{t+1}=b_t+n+1$.

For $k\in \mathbb{N}$, let $B_k=\{b_k, b_k+1, b_k+2, \ldots b_{k+1}-1\}$. Finally, let 
\[A=\{x\in \mathbb{N}\mid (\forall k\in \mathbb{N})\left(B_k \setminus Supp(x)\neq \phi\right)\}\]
and let $A^\prime=A\cup \{0\}$.

From \cite[Theorem 2.1]{key-14}, the upper Banach density of $A$ is $0$.

The proof that $A$ is a $J_{p}$-set is verbatim that of \cite[Theorem 2.1]{key-14}. For the sake of completeness, here we give the proof.

 For $k\in \omega$, let $B_k=\{b_k, b_k+1, b_k+2, \ldots b_{k+1}-1\}$. Let $A=\{x\in \mathbb{N}:
(\forall k\in \omega)(B_k \setminus supp(x)\neq \phi)\}$,
$A^\prime=\{x\in \omega: (\forall k\in \omega)(B_k \setminus supp(x)\neq \phi)\}$ and so, $A^\prime=A\cup \{0\}$.
From \cite[Theorem 2.1]{key-14}, $\bar d(A)=0$. Now assume $S\in \wp_{fin}(\mathbb{N})$, $F$ be finite collection of $IP$ sets, and $r= \lvert F\rvert$. Pick $k$ such that
$b_{k+1}-b_k>r$ and $b_k>n$. Pick $H\in \mathcal{P}_f(\mathbb{N})$ such that $\min H>m$ and for all
$f\in F$, $P(\sum_{t\in H}f(t))\in \mathbb{Z}2^{b_k} \, \forall P\in S$, this is possible as all $P$ has zero constant term.

Now pick $c\in \mathbb{N}2^{b_k}$ such that $\forall f\in F, P\in S, c+P(\sum_{t\in H}f(t))>0, \forall f\in F, \forall P\in S$.
Let $l= \max (\bigcup \{supp (c+P(\sum_{t\in H}f(t)): f\in F, P\in S\})$ and pick $j$ such that $l<b_j$. Pick $r_0\in B_k$ such that $B_k\setminus supp \left(2^{r_0}+c+P(\sum_{t\in H}f(t))\right)\neq \emptyset$ for each $f\in F$ and $P\in S$. Inductively for $i\in \lbrace 1,2,\ldots j-k\rbrace$, pick $r_i\in B_{k+i}$ such that $B_{k+i}\setminus supp \left(2^{r_i}+\sum_{t=0}^{i-1}2^{r_t}+c+P(\sum_{t\in H}f(t))\right)$ for each $f\in F$ and $P\in S$. Let $d=c+\sum_{i=0}^{j-k}2^{r_i}$.

As every set with positive Banach density is a $J_{p}$ set, the family of $J_{p}$-sets is partition regular. Henceforth, it makes sense to study the family of ultrafilters it contains.

\begin{defn}
 We set $\mathcal{J}_{p}=\left\{ p\in\beta\mathbb{N}\mid\forall A\in p \ A\in J_{p}\right\}$.
\end{defn}

As all piecewise syndetic sets are $J_{p}$ sets, it immediately follows that $\overline{K\left(\beta\mathbb{N},+\right)}\subseteq \mathcal{J}_{p}$. More in general, the following routine result (that we will prove by nonstandard means in Section 3) holds.

\begin{thm}\label{ideal}
$\mathcal{J}_{p}$ is a two sided ideal of $(\beta\mathbb{N},+)$.
\end{thm}

Less routine are some multiplicative properties of $\mathcal{J}_{p}$ that we will discuss at the end of this Section.

By Ellis' theorem \cite[Corollary 2.39]{key-6}, a straightforward consequence of Theorem \ref{ideal} is that there are idempotent ultrafilters, and even minimal idempotent ultrafilters, in $\mathcal{J}_{p}$. We denote the set of all idempotents in $\mathcal{J}_{p}$ by $E\left(\mathcal{J}_{p}\right)$. 

A long studied family of sets in the literature are $C$-sets, namely $J$-sets that belongs to some idempotent made of $J$-sets. In complete analogy, in our setting it makes sense to introduce the following polynomial version of $C$-sets.

\begin{defn}
$A\subseteq\mathbb{N}$ is a $C_{p}$-set if $A\in p$ for some idempotent ultrafilter $p\in E\left(\mathcal{J}_{p}\right)$ \footnote{$E\left(\mathcal{J}_{p}\right)$ is the set of idempotent ultrafilters in $\mathcal{J}_{p}$.}.
\end{defn}

By the definition we see immediately that all central and all $C$-sets are $C_{p}$-sets. $C_{p}$ sets play a fundamental role in the following theorem, which is the main result of our article. This result establishes the polynomial extension of the stronger Central Sets Theorem of De, Hindman and Strauss.

To simplify notations, from now we will denote the set of all sequences over the naturals $\mathbb{N}$ by $\tau$.

\begin{thm}\label{main} Let $A$ be a $C_{p}$-set and $S\in\wp_{fin} \left(\mathbb{P}\right)$ be finite.
Then there exists $\alpha:\wp_{fin}(\tau)\rightarrow \mathbb{N}, H:\wp_{fin}(\tau)\rightarrow \wp_{fin}(\mathbb{N})$ such that

\begin{enumerate}
\item\label{(1)} if $F,G\in \wp_{fin}(\tau),F\subset G$, then $\max H(F)<\min H(G)$;

\item\label{(2)} if $n\in\mathbb{N}$ and $G_{1},G_{2},\ldots,G_{n}\in \wp_{fin}(\tau),G_{1}\subsetneq G_{2}\subsetneq\cdots\subsetneq G_{n}$ and
$f_{i}\in G_{i},i=1,2,...,n$, then

\[
\sum_{i\in \beta } \alpha(G_{i})+P\left(\sum_{i\in \beta}\sum_{t\in H(G_{i})}f_{i}(t)\right)\in A.
\]
for all $\beta\subseteq \{1,2,\ldots,n\}$
\end{enumerate}
\end{thm}

\begin{proof} 
Choose an idempotent $p\in\mathcal{J}_{p}$ with $A\in p$. For $F\in \wp_{fin}(\tau)\setminus\{\emptyset\}$, we define $\alpha(F)\in\mathbb{N}$ and $H(F)\in\wp_{fin}\left(\mathbb{N}\right)\setminus\{\emptyset\}$ witnessing $(\ref{(1)}),(\ref{(2)})$ by induction on $|F|$.

For the base case of induction, let $F=\{f\}$. As $p$ is idempotent, the set $A^{\star}=\{x:-x+A\in p\}$ belongs to $p$, hence it is a $J_{p}$ set. So there exist $\beta\in\wp_{fin}\left(\mathbb{N}\right)\setminus\{\emptyset\}$
and $a\in\mathbb{N}$ such that 

\[
\forall P\in S,  \ a+P\left(\sum_{t\in\beta}f(t)\right)\in A^{\star}.
\]

By setting $\alpha(\{f\})=a$ and $H(\{f\})=\beta$, conditions $(\ref{(1)}),(\ref{(2)})$ are satisfied.

Now assume that $|F|>1$ and $\alpha(G)$ and $H(G)$ have been defined for all proper subsets $G$ of $F$. Let $K=\bigcup\{H(G):\emptyset\neq G\subset F\}\in \wp_{fin}(\mathbb{N})$, $m=\max K$ and

\begin{multline*}
M=\{\sum_{i=1}^{n} \alpha(G_{i})+P\left(\sum_{i=1}^{n}\sum_{t\in H(G_{i})}f_{i}(t)\right)\mid n\in\mathbb{N},\emptyset\neq G_{1}\subsetneq G_{2}\subsetneq\cdots\subsetneq G_{n}\subset F,\\ f_{i}\in G_{i},\forall i=1,2,..,n,\,P\in S \}.
\end{multline*}

Let
\[
B=A^{\star}\cap(\bigcap_{x\in M}(-x+A^{*}))\in p.
\]

For any $a\in\mathbb{N}$ , $\emptyset\neq G_{1}\subsetneq G_{2}\subsetneq\cdots\subsetneq G_{n}\subsetneq F$,
$f_{i}\in G_{i},$ for all $i=1,2,..,n$, and $P\in S$, let us define the polynomial 

\[
\phi_{\left\langle f_{i}\right\rangle _{i=1,}^{n},y}^{P,n}(y)=P\left(y+\sum_{i=1}^{n}\sum_{t\in H(G_{i})}f_{i}(t)\right)-P\left(\sum_{i=1}^{n}\sum_{t\in H(G_{i})}f_{i}(t)\right).
\]

Let $S^{\prime}=S\cup\left\{ \phi_{\left\langle f_{i}\right\rangle _{i=1,}^{n},y}^{P,n}\mid n<|F|,\mathcal{J}=\emptyset\neq G_{1}\subsetneq G_{2}\subsetneq\cdots\subsetneq G_{n}\subsetneq F\right\}$ and, for $f\in F$ and $\gamma\in\wp_{fin}\left(\mathbb{N}\right)\setminus\{\emptyset\}$, define $f\left(\gamma\right)=\sum_{i\in\gamma}f\left(i\right)$.

From Theorem \ref{recurrence}, there exists $\gamma\in\wp_{fin}\left(\mathbb{N}\right)\setminus\{\emptyset\}$ with $\min(\gamma)>m$ and $a\in\mathbb{N}$ such that 

\[
\forall P\in S,f\in F \ a+P(f(\gamma))\in B.
\]
 
We set $\alpha(F)=a$ and $H(F)=\gamma$. Now define $H(F)=\gamma$ and $\alpha(F)=a$.

Now, choosing $x=\sum_{i=1}^{n}(\alpha(G_{i})+P(\sum_{i=1}^{n}\sum_{t\in H(G_{i})}f_{i}(t))$
we have,

\begin{equation*}
  \begin{gathered}a+\phi_{\langle f_{i}\rangle_{i=1,}^{n},y}^{P,n}(f(\gamma))=\\
a+P\left(f(\gamma)+\sum_{i=1}^{n}\sum_{t\in H(G_{i})}f_{i}(t)\right)-P\left(\sum_{i=1}^{n}\sum_{t\in H(G_{i})}f_{i}(t)	\right)\in-x+A^{\star},\end{gathered}
  \end{equation*}
and so 
\begin{multline*} x+a+P\left(f(\gamma)+\sum_{i=1}^{n}\sum_{t\in H(G_{i})}f_{i}(t)\right)-P\left(\sum_{i=1}^{n}\sum_{t\in H(G_{i})}f_{i}(t)\right)=\\
=\left(\sum_{i=1}^{n+1}\alpha(G_{i})\right)+P\left(\sum_{i=1}^{n+1}\sum_{t\in H(G_{i})}f_{i}(t)\right)\in A^{*},\end{multline*}

where in the last line we let $F=G_{n+1}$.

This completes the induction argument, hence the proof.\end{proof}

Notice that now Theorem \ref{gcst} can be seen as a particular case of the above theorem obtained by taking $P(x)=x$.

By observing that any $\beta\in\wp_{fin}\left(\mathbb{N}\right)\setminus\{\emptyset\}$ is a subset of $\{1,\dots,n\}$ for some large enough $n\in\mathbb{N}$, we deduce the following seemingly stronger version of Theorem \ref{main}.

%FORSAYAN: You have to make a decision. I erased all references to arbitrary semigroups, but things keep oscillating between $\mathbb{N}$ and $\mathbb{Z}$. Now either: (1) we aim for a very general version and we go with semigroups, or (2) PICK ONE AND STICK WITH IT EVERYWHERE. Some of the oscillations where introduced by me, I tried to write everything for $\mathbb{N}$ (if we are not aiming for full generality, that is already good enough) but now I am lost on what is what.

\begin{cor}
Let $A$ be a $C_{p}$-set and $S\in \wp_{fin}(\mathbb{P})$. Then there exist $\alpha:\wp_{fin}(\tau)\rightarrow \mathbb{N}, H:\wp_{fin}(\tau)\rightarrow \wp_{fin}(\mathbb{N})$ such that

\begin{enumerate}
\item if $F,G\in \wp_{fin}(\tau),F\subset G$, then $\max H(F)<\min H(G)$;

\item if $\langle G_{n} \rangle_{n\in \mathbb{N}}$ is a sequence in $\wp_{fin}(\tau)$ such that $G_{1}\subsetneq G_{2}\subsetneq\cdots\subsetneq G_{n} \subsetneq \cdots$ and
$f_{i}\in G_{i},i\in \mathbb{N}$, then

\[
\sum_{i\in \beta } \alpha(G_{i})+P\left(\sum_{i\in \beta}\sum_{t\in H(G_{i})}f_{i}(t)\right)\in A.
\]
for all $\beta \in \wp_{fin}\left(\mathbb{N}\right)\setminus\{\emptyset\}$.
\end{enumerate}
\end{cor}

Finally, we want to show that Theorem \ref{main} can be extended to polynomials with rational coefficients, which has some interesting consequences regarding the multiplicative structure of $\mathcal{J}_{p}$.

\begin{thm}
\label{rational}Let $l\in\mathbb{N}$ and $A\subset\mathbb{N}$ be
a $J_{P}$ set. For each $i=1,2,\ldots,l,$ let $\left(x_{\alpha}^{i}\right)_{\alpha\in\wp_{fin}\left(\mathbb{N}\right)\setminus\{\emptyset\}}$
be an IP-set. Then for any $F\in\wp_{fin}\left(\mathbb{P}\right)$
there exist $a\in\mathbb{Z}$, $\beta\in\wp_{fin}\left(\mathbb{N}\right)\setminus\{\emptyset\}$ such that 
\[
a+P(x_{\beta}^{i})\in A
\]
 for all $i\in\{1,2,\ldots,l\}$ and $P\in F$.
\end{thm}

\begin{proof}
Let $M\in\mathbb{N}$ be the smallest common multiple of all denominators that appear among the coefficients of the polynomials in $F$. For each $P\in F$ and $n\in\mathbb{N}$, let $b_{n}$ be the coefficient of $x^{n}$ in $P$. We let $P^{\prime}$ be the polynomial obtained from
$P$ by multiplying each $b_{n}$ by $M^{n}$. With this construction, $P^{\prime}$ is a polynomial with integer coefficients. We let 

\[F^{\prime}=\left\{ P^{\prime}\mid P\in F\right\}.\]

Given an IP-set $\left(x_{\alpha}^{1}\right)_{\alpha\in\wp_{fin}\left(\mathbb{N}\right)\setminus\{\emptyset\}}$,
one can use the pigeonhole principle to choose a sequence $\langle H_{n}^{1}\rangle_{n\in\mathbb{N}}$
in $\wp_{fin}\left(\mathbb{N}\right)\setminus\{\emptyset\}$ so that $x_{\alpha}^{1}$ is a multiple of $M$ for each $\alpha\in \langle H_{n}^{1}\rangle_{n\in\mathbb{N}}$
and $n\in\mathbb{N}$. Now again applying the pigeonhole principle
over the $IP$ set $FU(\langle H_{n}^{1}\rangle_{n\in\mathbb{N}})$, we
obtain an another sequence $\langle H_{n}^{2}\rangle_{n\in\mathbb{N}}$
in $FU\left(\langle H_{n}^{1}\rangle_{n\in\mathbb{N}}\right)$ such
that $M\vert x_{\alpha}^{2}$ for each $\alpha\in \langle H_{n}^{2}\rangle_{n\in\mathbb{N}}$ and  $n\in\mathbb{N}$. As there are $l$ sequences, after $l$ steps, this process will terminate. Then we end up with an $IP$
set $I=FU\left(\langle K_{n}\rangle_{n\in\mathbb{N}}\right)$ in $\wp_{fin}\left(\mathbb{N}\right)\setminus\{\emptyset\}$
such that $M\vert x_{\beta}^{i}$ for all $\beta\in I$ and $i\in\{1,2,\ldots,l\}$.

Now fix this $IP$ set $I$ and for each $\left(x_{\alpha}^{i}\right)_{\alpha\in\wp_{fin}\left(\mathbb{N}\right)\setminus\{\emptyset\}}$,
consider the subsystems $\left(y_{\beta}^{i}\right)_{\beta\in\wp_{fin}\left(\mathbb{N}\right)\setminus\{\emptyset\}}$,
where $y_{n}^{i}=\sum_{t\in K_{n}}x_{t}^{i}$. Hence $M\vert y_{\beta}^{i}$
for all $i=1,2,\ldots,l$ and $\beta\in\wp_{fin}\left(\mathbb{N}\right)\setminus\{\emptyset\}$. Now for each
$i\in\{1,2,\ldots,l\}$, let $\widetilde{x}_{\alpha}^{i}=\frac{y_{\alpha}^{i}}{M}$.
Take the finite set of $IP$ sets $\left(\widetilde{x}_{\alpha}^{i}\right)_{\alpha\in\wp_{fin}\left(\mathbb{N}\right)\setminus\{\emptyset\}}$
for all $i\in\{1,2,\ldots,l\}$.

For this new finite set of $IP$ sets $\left(\widetilde{x}_{\alpha}^{i}\right)_{\alpha\in\wp_{fin}\left(\mathbb{N}\right)\setminus\{\emptyset\}}$,
and finite set of polynomials $F'$, by Theorem \ref{main} there exists $a\in\mathbb{Z}$
and $\beta\in\wp_{fin}\left(\mathbb{N}\right)\setminus\{\emptyset\}$ such that $a+P'(\widetilde{x}_{\beta}^{i})\in A$
for all $i\in\{1,2,\ldots,l\}$ and $P'\in F'$. As for each $n\in\mathbb{N}$,
$i\in\{1,2,\ldots,l\}$ , the $n^{th}$ monomial is of the form 
\[
a_{n}M^{n}\left(\widetilde{x}_{\beta}^{i}\right)^{n}=a_{n}M^{n}\left(\frac{y_{\alpha}^{i}}{M}\right)^{n}=a_{n}\left(y_{\alpha}^{i}\right)^{n}=a_{n}\left(\sum_{t\in\cup_{n\in\alpha}K_{n}}x_{t}^{i}\right)^{n},
\]
 so it is the $n$-th monomial of $P$. Hence for $\beta=\bigcup_{n\in\alpha}K_{n}$ and $a\in\mathbb{Z}$ we have that 
\[
a+P\left(x_{\beta}^{i}\right)\in A
\]
 for all $i\in\{1,2,\ldots,l\}$ and $P\in F$, as desired.
\end{proof}

The aforementioned theorem gives us the following multiplicative property of $J_p$ sets.

%FORSAYAN: to make it readable, adjustments have to be done here. 

\begin{cor}
\label{prodjpset} If $A\subset\mathbb{N}$ is a $J_{P}$ set and
$n\left(\neq0\right)\in\mathbb{Z}$, then $n\cdot A$ is a $J_{P}$
set.
\end{cor}

\begin{proof}
Let $F\in \wp_{fin}({\mathbb{P}})$ and $\left(x_{\alpha}^{i}\right)_{\alpha\in\wp_{fin}\left(\mathbb{N}\right)\setminus\{\emptyset\}}$
be IP-sets for all $i=1,2,\ldots,l.$ Let
\[F^{\prime}=\left\{ \frac{1}{n}P\mid P\in F\right\}\].
As $A$ is a $J_{P}$ set, by Theorem \ref{rational} we find $a\in\mathbb{Z}$,
$\beta\in\wp_{fin}\left(\mathbb{N}\right)\setminus\{\emptyset\}$ such that $a+\frac{1}{n}P(x_{\beta}^{i})\in A$
for all $i\in\{1,2,\ldots,l\}$ and $P\in F$, which implies that $na+P\left(x_{\beta}^{i}\right)\in n\cdot A$.

\end{proof}

\begin{cor}
If $A\subset\mathbb{N}$ is a $C_{p}$-set, then $n^{-1}\cdot A$ are $C_{P}$-sets.
\end{cor}

\begin{proof} Let $A\subset\mathbb{Z}$ be a $C_{P}$ set and let $A\in q$ for some $q\in E\left(\mathcal{J}_{p}\right)$. First of all, observe that $nq, n^{-1}q$ are idempotent.

$nA$ is a $C_{p}$-set as it belongs to $nq$, which is an idempotent made of $J_{p}$-sets. In fact, by definition $B\in nq$ if and only if $B\supseteq nB^{\prime}$ for some $B^{\prime}\in q$, and since $B^{\prime}$ is a $J_{p}$-set also $B$ is by Corollary \ref{prodjpset}.

As $n^{-1}A\in n^{-1}q$ and $n^{-1}q$ is an idempotent, if we can show that each element of $n^{-1}q$ is a $J_p$ set, we will have $n^{-1}A$ is a $C_p$ set.
 
Suppose $B\in n^{-1}q$, then we have $n\cdot B\in q$.
So, for any finite $F\in \wp_{fin}(\mathbb{P})$ and for any $l$ ($l\geq 1$) $IP$ sets, $\left(x_{\alpha}^{i}\right)_{\alpha\in\mathcal{F}}$, we have from theorem \ref{rational}, $a+n\cdot P(x_{\beta}^{i})\in n\cdot B$
for some $a\in\mathbb{Z}$, $\beta\in\mathcal{F}$ for all $P\in F\cup\{0\},$
where the polynomial $0$ vanishes over $\mathbb{N}.$ Hence $a\in n\cdot B$.
This implies $\frac{a}{n}+P(x_{\beta}^{i})\in B$ and so $B$ is a $J_{p}$
set.

This completes the proof.
\end{proof}

As a trivial consequence, we obtain the following multiplicative property of $\mathcal{J}_{p}$.
\begin{cor}
$\mathcal{J}_{p}$ is a left ideal of $(\beta\mathbb{N},\cdot)$.
\end{cor}

\section{Nonstandard versions}\label{nonstandard}

In the past few years, nonstandard analysis has played an important role in many developments to Ramsey theory. We refer to \cite{key-7} for an introduction to nonstandard methods tailored for applications in Ramsey theory. Fundamental for these developments are the nonstandard translations of combinatorial definitions of sets and ultrafilters, which often end up simplifying the development of several applications. A nonstandard take on central, $J-$ and $C$-sets, as well as a nonstandard proof of the Central Sets Theorem of Furstenberg can be found in \cite{key-20}. In the same veins, we want to find similar nonstandard characterization of $J_{p}$- and $C_{p}$-sets here, as well as prove some properties of $\mathcal{J}_{p}$.

Solely in this section, we assume the reader to be familiar with the basics of nonstandard analysis. For our purposes, it is sufficient to recall that, given $ p $ ultrafilter on $\N$, the set 
\[\mu( p ):=\left\{\alpha\in\,^{\ast}\N\mid \forall A\in p  \ \alpha\in\,^{\ast}A\right\}\]

is called the monad of $ p $; as shown in detail in \cite{key-21}, many combinatorial properties of $ p $ correspond to combinatorial properties of its monad.

Conversely, given $\alpha\in\,^{\ast}\N$, we let 
\[ p _{\alpha}:=\left\{A\subseteq \N\mid \alpha\in\,^{\ast}A\right\}.\]

It is immediate to prove that for all $\alpha\in ^{\ast}\N$ $ p _{\alpha}\in\beta\N$; conversely, if $^{\ast}\N$ is at least $\vert \wp(\mathbb{N})\vert^{+}$-saturated\footnote{We will assume this from now on.}, $\mu( p )\neq\emptyset$ for all $ p \in\beta\N$ filters on $\N$.

The nonstandard characterization of operations between ultrafilters is more complicated, and it is often done in a setting that allows for iterated hyperextensions (see e.g. \cite{key-7}). As here we will make a very limited use of these nonstandard characterization of operations, we will just recall that, in general, it is not true that $ p _{\alpha}\oplus p _{\beta}= p _{\alpha+\beta}$. When this happens, we will say that $(\alpha,\beta)$ is a tensor pair. Notably, for any $ p ,q\in\beta\mathbb{N}$ there exists $\alpha\in\mu( p ),\beta\in\mu(q)$ such that $(\alpha,\beta)$ is a tensor pair. We refer to \cite{key-21} for a detailed study of tensor pair, monads and their combinatorial applications.

The nonstandard characterization of $J_{p}$ sets can be obtained via a routine enlargement argument.

\begin{prop}\label{easy} Let $A\subseteq\N$ and $\mathcal{IP}=\{A\subseteq\mathbb{N}\mid A \ \text{is IP-rich}\}$. The following facts are equivalent:

\begin{enumerate} 
\item $A$ is a $J_{p}$-set;
\item there exists $\eta\in\,^{\ast}\N,\beta\in\,^{\ast}\wp_{fin}\left(\mathbb{N}\right)\setminus\{\emptyset\}$ such that for all IP-sets $(x_{\alpha})_{\alpha\in\wp_{fin}\left(\mathbb{N}\right)\setminus\{\emptyset\}}$ 

\[\forall P\in\mathbb{P}\  \eta + P\left(x_{\beta}\right)\in\,^{\ast} A;\]
\item there exists $\eta\in\,^{\ast}\N,\beta\in\,^{\ast}\wp_{fin}\left(\mathbb{N}\right)\setminus\{\emptyset\}, I\in\,^{\ast}\wp_{fin}(\mathcal{IP})$, $H\in\,^{\ast}\wp_{fin}(\mathbb{P})$ such that $\mathcal{IP}\subset I, \mathbb{P}\subseteq H$ and

\[\forall (x_{\alpha})_{\alpha\in\,^{\ast}\wp_{fin}\left(\mathbb{N}\right)\setminus\{\emptyset\}}\in I, \forall P\in H\  \eta + P\left(x_{\beta}\right)\in\,^{\ast} A.\]
\end{enumerate}
\end{prop}

\begin{proof} In the proof, for all $F\in\wp_{fin}(\mathbb{P}), G\in\wp_{fin}(\mathcal{IP})$ we  let 
\[A_{F,G}:=\left\{(a,\beta)\in \N\times \wp_{fin}\left(\mathbb{N}\right)\setminus\{\emptyset\} \mid \forall P\in F, (x_{\alpha})_{\alpha\in\wp_{fin}\left(\mathbb{N}\right)\setminus\{\emptyset\}} \ a+P(x_{\beta})\in A\right\}.\]

$(1)\Rightarrow (2)$ Assume that $A$ is a $J_{p}$-set. As $A\in J_{p}$, $A_{F,G}\neq\emptyset$; moreover, the family $\{A_{F,G}\}_{F,G}$ has trivially the finite intersection property. By enlargement, $\bigcap_{F,G} \,^{\ast}A_{F,G}\neq\emptyset$, and any pair $(\eta,\beta)$ in this intersection satisfies our conclusion.

$(2)\Rightarrow (3)$ This is a routine argument: for all $F,G\in\ \wp_{fin}(\mathbb{P})\times \wp_{fin}(\mathcal{IP})$, the set 
\begin{multline} S_{F,G}:=\{\left(\widetilde{F},\widetilde{G}\right)\in\,^{\ast}\left(\wp_{fin}(\mathbb{P})\times \wp_{fin}(\mathcal{IP})\right)\mid F\subseteq\widetilde{F}, G\subseteq\widetilde{G} \ \text{and}\\ \exists \eta\in\,^{\ast}\N\ \exists\beta\in\,^{\ast}\wp_{fin}\left(\mathbb{N}\right)\setminus\{\emptyset\}\  \forall (x_{\alpha})_{\alpha\in\,^{\ast}\wp_{fin}\left(\mathbb{N}\right)\setminus\{\emptyset\}}\in \widetilde{G}\ \forall P\in \widetilde{F}\  \eta + P\left(x_{\beta}\right)\in\,^{\ast} A\}\end{multline}
is non empty. As the family $\{S_{F,G}\}_{F,G}$ has the finite intersection property, by saturation $\bigcap_{F,G} S_{F,G}\neq \emptyset$. Any $(H,I)$ in this intersection proves our thesis.

$(3)\Rightarrow (1)$ As $\mathbb{P}\subseteq H, \mathcal{IP}\subseteq I$, the hypothesis ensures that, for all $F\subseteq \mathbb{P}$ and $G$ finite family of IP sets, $^{\ast}A_{F,G}\neq\emptyset$. Hence, $A_{F,G}\neq \emptyset$ by transfer. Any $(a,\beta)\in A_{F,G}$ proves that $A\in J_{p}$. \end{proof}

Similarly, we can obtain a characterization for ultrafilters in $\mathcal{J}_{p}$.

\begin{prop}\label{easy2} Let $ p \in\beta\N$. The following facts are equivalent:
\begin{enumerate}
\item $ p \in\mathcal{J}_{p}$;

\item there exists $\eta\in\,^{\ast}\N,\beta\in\,^{\ast}\wp_{fin}\left(\mathbb{N}\right)\setminus\{\emptyset\} $ such that for all IP-sets $(x_{\alpha})_{\alpha\in\wp_{fin}\left(\mathbb{N}\right)\setminus\{\emptyset\}}$ 

\[\forall P\in\mathbb{P}\  \eta + P\left(x_{\beta}\right)\in\mu( p );\]

\item  there exists $\eta\in\,^{\ast}\N,\beta\in\,^{\ast}\wp_{fin}\left(\mathbb{N}\right)\setminus\{\emptyset\}, I\in\,^{\ast}\wp_{fin}(\mathcal{IP})$, $H\in\,^{\ast}\wp_{fin}(\mathbb{P})$ such that $\mathcal{IP}\subset I, \mathbb{P}\subseteq H$ and

\[\forall P\in H\  \eta + P\left(x_{\beta}\right)\in\mu( p ).\]
\end{enumerate}
\end{prop}

\begin{proof} $(1)\Rightarrow (3)$ By Proposition \ref{easy}, for all $A\in p $, for all $F,G\in\wp_{fin}(\mathcal{IP})\times\wp_{fin}(\mathbb{P})$ the set 
\begin{multline} I_{A,F,G}=\{(\eta,\beta,I,H)\in\,^{\ast}\left(\N\times\wp_{fin}\left(\mathbb{N}\right)\setminus\{\emptyset\}\times\wp_{fin}(\mathcal{IP})\times\wp_{fin}(\mathbb{P})\right)\mid\\
^{\ast}F\subseteq I,\,^{\ast}G\subseteq H\ \text{and} \ \forall (x_{\alpha})_{\alpha\in\,^{\ast}\wp_{fin}\left(\mathbb{N}\right)\setminus\{\emptyset\}}\in I, \forall P\in H\  \eta + P\left(x_{\beta}\right)\in\,^{\ast} A\}\neq\emptyset.\end{multline}
All $I_{A,F,G}$ are internal and the family \[\{I_{A,F,G}\mid A\in p , F,G\in\wp_{fin}(\mathcal{IP})\times\wp_{fin}(\mathbb{P})\}\] has the finite intersection property hence, by saturation, $\bigcap_{A,F,G} I_{A,F,G}\neq\emptyset$. Any $(\eta,\beta,I,H)$ witnesses the validity of (3).

$(3)\Rightarrow (2)$ This is immediate.

$(2)\Rightarrow (1)$ As $\mu( p )\subset\,^{\ast}A$ for any $A\in p $, we conclude by Proposition \ref{easy}. \end{proof}

As a trivial consequence, we can give a nonstandard proof of Theorem \ref{ideal}.

\begin{proof} Let $ p \in\mathcal{J}_{p},q\in\beta\N$. Let $(\eta,\beta,I,H)$ be given as in condition (3) of Proposition \ref{easy2}. If $\sigma\in\mu(q)$ is such that $(\eta+P\left(x_{\beta}\right),\sigma)$ is a tensor pair for all $P\in H, (x_{\alpha})\in I$, we just have to observe that $(\eta+\sigma,\beta,I,H)$ is such that $\mathcal{IP}\subset I, \mathbb{P}\subseteq H$ and

\[\forall P\in H\  \eta + \sigma + P\left(x_{\beta}\right)\in\mu(p \oplus q ).\]

Similarly, if $\sigma\in\mu(q)$ is such that $(\sigma,\eta+P\left(x_{\beta}\right))$ is a tensor pair for all $P\in H, (x_{\alpha})\in I$, we just have to observe that $(\sigma+\eta,\beta,I,H)$ is such that $\mathcal{IP}\subset I, \mathbb{P}\subseteq H$ and

\[\forall P\in H\  \sigma+\eta + P\left(x_{\beta}\right)\in\mu(p \oplus q ).\qedhere\] \end{proof}

As in the case of $J$ and $C$ sets, the characterizations of $J_{p}$ sets and $\mathcal{J}_{p}$ ultrafilters can be extended to $C_{p}$ sets and idempotent ultrafilters in $\mathcal{J}_{p}$ just by recalling that $ p \in\beta\N$ is idempotent if and only if there are $\alpha,\beta\in\mu( p )$ with $(\alpha,\beta)$ tensor pair and $\alpha+\beta\in\mu( p )$, so that a set $A\in\wp(\N)$ is contained in an idempotent if and only if it contains $\alpha,\beta\in\mu( p )$ with $(\alpha,\beta)$ tensor pair and $\alpha+\beta\in\mu( p )$. Henceforth, a set $A$ is a $C_{p}$ set if it satisfies any of the equivalent properties listed in Proposition \ref{easy} and it contains $\alpha$ that generates an idempotent.

\section{Discussions and further possibilities}

The introduction of $J_{p}$ and $C_{p}$ sets rises many questions. We want to list some of them here, providing some comments on why we believe these are relevant.

\begin{qn} Is the family of $J_{p}$-sets strongly partition regular?\end{qn}

We have not been able to prove this fact; simple modifications of the proof of the same result for $J$-sets seems not to work. Anyhow, we do believe that the answer to the above question is positive, a reason being that it is possible to prove that the related family of PP-rich sets is strongly partition regular. Let us first recall its definition.

\begin{defn}
 $A\subseteq \mathbb{N}$ is PP-rich if for each $F\in \wp_{fin}(\mathbb{P})$ there exist
 $a, x\in \mathbb{N}$ such that $a+P(x)\in A$ for all $P\in F$.
\end{defn}

\begin{thm}
 The family of PP-rich sets is strongly partition regular.
\end{thm}

\begin{proof}
Let $A$ be a PP-rich set, and let $A=A_1\cup A_2$. By contrast, let us assume that $A_1$ and $A_2$ are not PP-rich sets. Let $A_1$ does not contain the polynomial progression of the finite set of polynomials $F_1\in \wp_{fin}(\mathbb{P})$, and $A_2$ does not contain polynomial progression of the finite set of polynomials $F_2\in \wp_{fin}(\mathbb{P})$. 
 Let $F=F_1\cup F_2$, and 
\[n=\max\{degP: P\in F\}, l=max\{Coef(P):P\in F\},\] 
where $Coef(P)$ is the maximum coefficient of polynomial $P$.

 By polynomial van der Wearden's Theorem and compactness argument, there exists a sufficiently large $N\in \mathbb{N}$ such
 that if $[1, N]^n$ is 2-colored, then one of the color classes contain a monochromatic structure of the form 
 
 \[
 (z_1+j_1w, z_2+j_2w^2, \ldots ,z_n+j_nw^n), 0\leq j_k\leq l \, \text{for} \, 1\leq k\leq n.
\]

 For $\vec a=(a_1, a_2, \ldots, a_n)\in \omega^n$, let us define the polynomoial $P_{\vec a}(y)=a_1y+a_2y^2+\ldots+a_ny^n$, where $\omega= \mathbb{N}\cup \{0\}$.
 Let us define
  \[
  G=\{P_{\vec a}: \vec a=(a_1, a_2, \ldots, a_n)\in \omega^n \, \text{with} \, 0\leq a_i\leq N \, and \, 1\leq i\leq n\}.
  \]
  
   Now choose
 $x, y\in \mathbb{N}$ such that $\{x+P(y): a\in G\}\subseteq A$. Color the set  $[1, N]^n=C_1\cup C_2$ as  $\vec a \in C_i$ if and
 only if $x+P_{\vec a}(y)\in A_i$. So there exist, $i\in \{1, 2\}$ and $z_1, z_2, \ldots, z_n\in \mathbb{N}$, $w\in \mathbb{N}$
 such that 
 \[
 (z_1+j_1w, z_2+j_2w^2, \ldots ,z_n+j_nw^n)\in C_i, 0\leq j_k\leq l.
 \]
 
 Hence, $x+P_{\vec a_j}(y)\in A_i$, where $\vec a_j=(z_1+j_1w, z_2+j_2w^2, \ldots ,z_n+j_nw^n)$, $0\leq j_k\leq l$ for $1\leq k\leq n$. So,
 \[
 x+(z_1+j_1w)y+ (z_2+j_2w^2)y^2+ \ldots +(z_n+j_nw^n)y^n\in A_i, for 0\leq j_k\leq l, 1\leq k\leq n
 \]
  and thus
  \[
 (x+z_1y+z_2y^2+\ldots+z_ny^n)+j_1(wy)+j_2(wy)^2+\ldots +j_n(wy)^n\in A_i,
 \]
  for $0\leq j_k\leq l$, $1\leq k\leq n$.

 In particular, $a+P(wy)\in A_i$, where $a=x+z_1y+z_2y^2+\ldots+z_ny^n$.
 Therefore, PP-rich sets are strongly partition regular.
 \end{proof}

A related important question is the following one:

\begin{qn} Does it exists a $J$ set that is not a $J_{p}$ set?\end{qn}
 
Again, we don't have an answer to the above question. The reason is that this question is much harder to answer than it seems. In fact, the following is a similar question that has now been open for some years:

\begin{qn} Are $J$ sets PP-rich? \end{qn}

Notice that $J_{p}$ sets are PP-rich, so if we would be able to prove that all $J$ sets are $J_{p}$ sets, the above question would be solved affirmatively. On the other end, the precise relationship between PP-rich sets and $J_{p}$-sets is still unknow, as the following question remains open.

\begin{qn} Are PP-rich sets also $J_{p}$-sets? \end{qn}

We believe that the answer to the previous question should be no. In fact, the similar linear question relating AP-rich sets and $J$-sets has been solved in \cite[Lemma 4.3]{key-2}, where the authors have demonstrated that there exists an AP-rich set which is not a $J$ set. However, it seems that the argument can not be easily lifted from the linear to the polynomial case.

Finally, maybe the most relevant open question that has to be mentioned is the following: 

\begin{qn} Is our polynomial extension of the Stronger Central Sets Theorem actually more general than Theorem \ref{PCST}? \end{qn}

It has been shown in \cite[Theorem 4.4]{key-5} that the stronger Central Sets Theorem for arbitrary semigroups is indeed stronger than the original Central Sets Theorem for semigroups by considering a special free semigroup. However, it is still an open question if this is true or not on $\mathbb{N}$, or on any countable Abelian group.

\end{document}